\documentclass[11pt]{amsart}

\usepackage[T1]{fontenc}
\usepackage{lmodern}
\usepackage{amsmath,amssymb,amsthm,mathtools}
\usepackage{booktabs}
\usepackage{graphicx}
\usepackage[margin=1in]{geometry}
\usepackage{microtype}
\usepackage[hidelinks]{hyperref}

\newtheorem{theorem}{Theorem}[section]
\newtheorem{proposition}[theorem]{Proposition}
\newtheorem{lemma}[theorem]{Lemma}

\theoremstyle{definition}
\newtheorem{definition}[theorem]{Definition}
\theoremstyle{remark}

\newcommand{\Z}{\mathbb{Z}}
\newcommand{\odim}{\operatorname{odim}}

\newcommand{\mset}[1]{\{\!\{#1\}\!\}}
\newcommand{\eps}{\varepsilon}

\title{The Outer Multiset Dimension of Toroidal Grids}
\author[B. Peng]{
  Bo Peng \\ \vspace{2pt}
  \textmd{\small Institute of Mathematical Sciences, ShanghaiTech University}
}
\date{September 16, 2026}

\hypersetup{
  pdftitle={The Outer Multiset Dimension of Toroidal Grids},
  pdfauthor={Bo Peng}
}

\begin{document}

\begin{abstract}
Let $S$ be a set of vertices in a connected graph $G$.  A vertex outside
$S$ is represented by the multiset of its distances to the vertices of $S$.
The outer multiset dimension $\odim(G)$ is the minimum cardinality of an $S$
for which these representations distinguish all vertices outside $S$.  We
determine $\odim(C_s\mathbin{\square}C_t)$ for all $s,t\geq 3$, answering a
problem of Klav\v{z}ar, Kuziak, and Yero.  The values range from $3$ to $8$.
The proof combines a half-turn argument giving a universal four-landmark
lower bound when both factors have length at least four, explicit three- and
four-landmark constructions for the infinite families, and exact finite
enumeration on the remaining strip.  The collision classification behind the
infinite four-landmark construction is certified by exact quantifier
elimination in linear integer arithmetic; source code and all finite upper
certificates accompany the paper.
\end{abstract}

\maketitle

\section{Introduction}

Let $G$ be a finite connected graph and let $S=\{w_1,\ldots,w_k\}$ be a set
of vertices.  The usual metric representation of a vertex records the
distances to the landmarks in a fixed order.  If the landmark identities are
discarded, the resulting representation is the multiset
\[
  m_G(x\mid S)=\mset{d_G(x,w):w\in S}.
\]
Gil-Pons, Ram\'{\i}rez-Cruz, Trujillo-Rasua, and Yero introduced the outer
multiset dimension in order to require these representations to distinguish
only the vertices outside $S$ \cite{GilPonsEtAl2019}.  This avoids the
nonexistence phenomenon that occurs for the non-outer multiset dimension.

Klav\v{z}ar, Kuziak, and Yero determined the outer multiset dimension of
Cartesian products of two paths and asked for the corresponding value on a
Cartesian product of two cycles \cite[Problem~6.4]{KlavzarKuziakYero2023}.
The question is repeated as an open problem in the recent survey of Farhan,
Klav\v{z}ar, Kuziak, and Yero \cite[Problem~11]{FarhanEtAl2026}.  We solve it
for every pair of cycle lengths.

Write
\[
  a=\min\{s,t\},\qquad b=\max\{s,t\}.
\]
The complete answer is the following.

\begin{theorem}\label{thm:main}
For all integers $s,t\geq 3$, the outer multiset dimension of
$C_s\mathbin{\square}C_t$ is as follows.

If $a=3$, then
\[
\odim(C_a\mathbin{\square}C_b)=
\begin{cases}
8,&b=3,\\
7,&b=5,\\
5,&b\in\{4,6,7,9,11\},\\
4,&b\in\{8,10\}\text{ or }b\geq13\text{ is odd},\\
3,&b\geq12\text{ is even}.
\end{cases}
\]

Suppose that $a\geq4$.  The value is $7$ at $(a,b)=(4,4)$ and is $6$ at
$(a,b)\in\{(5,5),(6,6)\}$.  It is $4$ precisely in the following cases:
\begin{enumerate}
  \item $b\in\{15,16\}$ or $b\geq18$;
  \item $b=17$ and $a\in\{4,5,6,8,10,12,14,15,16\}$;
  \item $(a,b)=(5,10)$;
  \item $b=12$ and $a\in\{4,5,6,8,10,12\}$;
  \item $b=14$ and $a\in\{4,5,6,8,10,12,14\}$.
\end{enumerate}
In every remaining case with $a\geq4$, the value is $5$.
\end{theorem}

Several features of the formula are worth noting.  Apart from three small
square exceptions, every torus with both factors of length at least four has
outer multiset dimension four or five.  Once the larger factor is at least
18, the value stabilizes at four.  The thin family
$C_3\mathbin{\square}C_t$ behaves differently: its value is three for every
even $t\geq12$ and four for every odd $t\geq13$.

The proof has three parts.  First, half-turns in the abelian group
$\Z_s\times\Z_t$ turn bisectors into pairs of vertices with the same
unordered distance representation.  This yields a four-landmark lower bound
for $s,t\geq4$.  Second, explicit four-sets resolve every sufficiently long
torus.  Their verification reduces to the collision structure of a
one-dimensional cyclic code.  Completeness of this collision classification
is an exact computer-assisted step over linear integer arithmetic.  Third,
the remaining finite strip is settled by exhaustive enumeration with
independently checked upper certificates.

The source accompanying this paper contains the exact affine certificate, the
finite enumerator, and all retained witnesses.  Section~\ref{sec:computation}
states precisely which conclusions use computation.

\section{Definitions and cyclic difference functions}

For $n\geq3$, identify $V(C_n)$ with $\Z_n$ and put
\[
  \delta_n(x)=\min\{\bar x,n-\bar x\},
  \qquad 0\leq\bar x<n,
\]
where $\bar x$ is the canonical representative of $x$.  Thus the distance on
the torus $T_{s,t}=C_s\mathbin{\square}C_t$ is
\begin{equation}\label{eq:torus-distance}
  d((i,j),(i',j'))=\delta_s(i-i')+\delta_t(j-j').
\end{equation}

\begin{definition}
A set $S\subseteq V(G)$ is an \emph{outer multiset resolving set} if
$m_G(x\mid S)\neq m_G(y\mid S)$ for all distinct
$x,y\in V(G)\setminus S$.  The minimum size of such a set is denoted
$\odim(G)$.
\end{definition}

Every outer multiset resolving set is an ordinary resolving set: equality of
the ordered distance vectors implies equality of the corresponding multisets,
and a landmark is distinguished from every nonlandmark by its zero coordinate.
Consequently,
\begin{equation}\label{eq:metric-lower}
  \dim(G)\leq\odim(G).
\end{equation}
We shall use the result of C\'{a}ceres et al. that
\begin{equation}\label{eq:ordinary-torus}
\dim(C_s\mathbin{\square}C_t)=
\begin{cases}
3,&\text{at least one of $s,t$ is odd},\\
4,&\text{both $s,t$ are even}
\end{cases}
\end{equation}
for $s,t\geq3$ \cite[Theorem~8.4]{CaceresEtAl2007}.

The following elementary difference function controls all bisectors used
below:
\[
  F(q,c,u)=\delta_q(u)-\delta_q(u-c),
  \qquad r=\delta_q(c).
\]

\begin{lemma}[cyclic difference table]\label{lem:difference-table}
The following statements hold.
\begin{enumerate}
  \item If $q$ is odd, the value set of $F(q,c,\cdot)$ is every integer in
  $[-r,r]$.  If $r>0$, zero occurs once, every value $k$ with
  $0<|k|<r$ occurs once, and each of $\pm r$ occurs
  $(q+1)/2-r$ times.
  \item If $q$ is even, the value set is
  $\{-r,-r+2,\ldots,r\}$.  If $r>0$, every nonendpoint value occurs twice,
  while each endpoint occurs $q/2-r+1$ times.  Zero occurs twice if $r$ is
  even and not at all if $r$ is odd.
  \item If $r=0$, the function is identically zero.
\end{enumerate}
In every case $F(q,c,c-u)=-F(q,c,u)$.
\end{lemma}

\begin{proof}
By reflecting the cycle, take $c=r$.  The involution $u\mapsto c-u$ pairs
the value $k$ with $-k$.  Along each of the two arcs joining $0$ to $c$, the
difference changes by two at each step until it reaches an endpoint plateau.
For odd $q$ the two resulting arithmetic progressions have opposite parity
and interlace, whereas for even $q$ they have the same parity and repeat the
interior values.  The lengths of the two plateaus give the endpoint
multiplicities below.

If $q=2\ell+1$ and $r>0$, then on
the four successive integer intervals the values are
\[
\begin{array}{c|c}
0\leq u\leq r-1&2u-r\\
r\leq u\leq\ell&r\\
\ell+1\leq u\leq\ell+r&2\ell+1+r-2u\\
\ell+r+1\leq u\leq2\ell&-r.
\end{array}
\]
For $q=2\ell$, the corresponding table is
\[
\begin{array}{c|c}
0\leq u\leq r-1&2u-r\\
r\leq u\leq\ell&r\\
\ell+1\leq u\leq\min\{\ell+r,2\ell-1\}&2\ell+r-2u\\
\ell+r+1\leq u\leq2\ell-1&-r.
\end{array}
\]
Empty intervals are omitted.  Reading the values and their multiplicities
from these tables proves the result.
\end{proof}

\section{The factor of length three}\label{sec:c3}

We begin with the infinite family where three landmarks suffice.

\begin{proposition}\label{prop:c3-even}
For every $m\geq6$,
\[
  \odim(C_3\mathbin{\square}C_{2m})=3.
\]
The set
\[
  S=\{(0,0),(0,4),(1,0)\}
\]
is an outer multiset basis.
\end{proposition}

\begin{proof}
The lower bound is \eqref{eq:metric-lower}--\eqref{eq:ordinary-torus}.  For a
column $j\in\Z_{2m}$, put
\[
  x=\delta_{2m}(j),\qquad y=\delta_{2m}(j-4),\qquad e=y-x.
\]
The values of $e$ are completely described by
\[
\begin{array}{c|c}
j&e\\ \hline
0,\ m+4,\ldots,2m-1&4\\
1,\ m+3&2\\
2,\ m+2&0\\
3,\ m+1&-2\\
4,\ldots,m&-4.
\end{array}
\]
For $m\geq6$, the pair $(x,e)$ determines $j$: in the three two-point rows
the two possible $x$-values are respectively $1,m-3$; $2,m-2$; and
$3,m-1$.

For a vertex $(r,j)$, subtracting $x$ from all three distances leaves one of
the offset multisets
\[
\begin{array}{c|c}
r&\text{offset multiset}\\ \hline
0&\mset{0,1,e}\\
1&\mset{0,1,e+1}\\
2&\mset{1,1,e+1}.
\end{array}
\]
After sorting and subtracting the smallest entry, the normalized shapes are
\[
\begin{array}{c|ccccc}
&e=-4&e=-2&e=0&e=2&e=4\\ \hline
r=0&(0,4,5)&(0,2,3)&(0,0,1)&(0,1,2)&(0,1,4)\\
r=1&(0,3,4)&(0,1,2)&(0,1,1)&(0,1,3)&(0,1,5)\\
r=2&(0,4,4)&(0,2,2)&(0,0,0)&(0,0,2)&(0,0,4).
\end{array}
\]
Only $(0,1,2)$ is repeated.  For a fixed $(r,e)$, equality of the original
multisets fixes their minimum, then $x$, and then $j$.  In the repeated-shape
case the row-zero multisets are
\[
  \mset{1,2,3},\quad \mset{m-3,m-2,m-1},
\]
whereas the row-one multisets are
\[
  \mset{2,3,4},\quad \mset{m-2,m-1,m}.
\]
No cross equality is possible for $m\geq6$; the only apparent endpoint
equation would force $m=5$.  Hence $S$ resolves.
\end{proof}

We next exclude three landmarks when the other factor is odd.

\begin{lemma}[odd torus bisector]\label{lem:c3-bisector}
Let $n\geq5$ be odd and let $G=\Z_n\times\Z_3$ with the Lee metric.  For
any $A,B\in G$, the bisector
\[
  \{x:d(x,A)=d(x,B)\}
\]
contains a nonfixed two-point orbit under $x\mapsto A+B-x$.
\end{lemma}

\begin{proof}
Translate $A$ to zero and write $B=(a,b)$.  The bisector equation is
\[
  F(n,a,u)+F(3,b,v)=0.
\]
If $b=0$ and $a\neq0$, the first function has one zero, while the two
nonzero $v$-coordinates form a nonfixed orbit.  If $a=0$ and $b\neq0$, use
the unique zero of $F(3,b,\cdot)$ and any nonzero pair $u,-u$.  If both
coordinates are nonzero, then $F(3,b,\cdot)$ takes $-1,0,1$, while
Lemma~\ref{lem:difference-table} shows that $F(n,a,\cdot)$ contains both
$-1$ and $1$.  The identity in that lemma pairs the two matched solutions.
The zero displacement is immediate.
\end{proof}

\begin{proposition}\label{prop:c3-odd-lower}
If $n\geq5$ is odd, then
\[
  \odim(C_3\mathbin{\square}C_n)\geq4.
\]
\end{proposition}

\begin{proof}
Sets of size at most two fail directly.  For a two-set $\{P,Q\}$, the map
$h(x)=P+Q-x$ swaps the two distances.  Since $|\Z_n\times\Z_3|$ is odd,
$h$ has one fixed point; apart from the possible orbit $\{P,Q\}$, it has
nonfixed orbits disjoint from the landmarks.

Let $S=\{P,Q,R\}$.  Put $R'=P+Q-R$ and let $h(x)=P+Q-x$.  If $x$ lies on
the bisector of $R,R'$ and $y=h(x)$, then the distances to $P,Q$ are
exchanged and the distance to $R$ is preserved.  A nonfixed such orbit
cannot contain $R$; it meets $P$ or $Q$ only when it is $\{P,Q\}$, which
requires $d(P,R)=d(Q,R)$.  If the three pairwise distances in $S$ are not
all equal, choose $R$ so that its distances to the other two differ and
apply Lemma~\ref{lem:c3-bisector}.

It remains to treat an equilateral triple of common distance $L$.  For this
paragraph, write the coordinates in the order
\[
  (\text{row},\text{column})\in\Z_3\times\Z_n.
\]
If all
landmarks occupy one row, two vertices in the empty rows and the same unused
column have identical representations.  If all three rows are occupied,
choose the landmarks in rows $0$ and $1$.  When their columns differ, the
unique equidistant column of the odd cycle differs from both landmark columns;
when they coincide, choose any unused column.  The vertices in rows $0$ and
$1$ at the chosen column exchange the first two distances.  Their distances
to the row-$2$ landmark are equal because both row distances are one.  Hence
the two vertices are outside $S$ and have the same representation.

In the remaining row pattern, normalize the landmarks to
\[
  (0,0),\quad(0,L),\quad(1,q),
  \qquad 1\leq L\leq\lfloor n/2\rfloor.
\]
Equilaterality gives
\[
  \delta_n(q)=\delta_n(q-L)=L-1.
\]
The first equation gives $q=L-1$ or $n-L+1$, and the second gives
$q=1$ or $2L-1$.  The intersections are precisely
\[
  (L,q)=(2,1),\qquad\text{or}\qquad n+2=3L,\ q=2L-1.
\]
If $n=2m+1$ in the first case, row-zero columns $m+1,m+2$ both have
representation $\mset{m-1,m,m+1}$.  In the second case write $L=2r+1$;
row-zero columns $r,r+1$ both have representation
$\mset{r,r+1,3r+1}$.  These vertices lie outside $S$.
\end{proof}

\section{A universal lower bound}\label{sec:lower}

We now prove the lower bound that drives the classification away from the
thin factor.

\begin{theorem}\label{thm:lower-four}
For all $s,t\geq4$,
\[
  \odim(C_s\mathbin{\square}C_t)\geq4.
\]
\end{theorem}

\begin{proof}
If both factors are even, this is immediate from
\eqref{eq:metric-lower}--\eqref{eq:ordinary-torus}.  Hence write the torus as
$G=\Z_n\times\Z_m$, where $n\geq5$ is odd and $m\geq4$.

Fix three landmarks $P,Q,R$ and define
\[
  h(x)=P+Q-x,\qquad R'=h(R).
\]
Every nonfixed orbit $\{x,h(x)\}$ in the bisector of $R,R'$ has the same
unordered distance triple to $P,Q,R$.  It never contains $R$, and if it
meets $P$ or $Q$ it is exactly $\{P,Q\}$.  It therefore suffices to find
two nonfixed bisector orbits for at least one choice of the third landmark.

After translating $R$ to zero, put
\[
  D_R=P+Q-2R=(a,b),\qquad
  \alpha=\delta_n(a),\quad\beta=\delta_m(b).
\]
The bisector equation is
\begin{equation}\label{eq:product-bisector}
  F(n,a,u)+F(m,b,v)=0.
\end{equation}
Suppose first that $m$ is even.  Lemma~\ref{lem:difference-table} shows that
\eqref{eq:product-bisector} has no nonfixed orbit exactly in the two cases
\begin{align*}
\mathrm{B0}:&\quad a=0\text{ and }\beta\text{ is odd},\\
\mathrm{B1}:&\quad \alpha=1\text{ and }\beta>0\text{ is even}.
\end{align*}
Indeed, in B0 the even-cycle difference has no zero, while the odd-cycle
difference is identically zero.  In B1 the only common value is zero and all
solutions are fixed by $h$.  Outside B0 and B1, the two value sets share a
nonzero absolute value, or zero has a free coordinate.

Moreover, if $m\geq6$, every displacement outside B0 and B1 supplies at
least two nonfixed orbits.  This follows directly from the multiplicities in
Lemma~\ref{lem:difference-table}:
\begin{itemize}
  \item if $\alpha=0$ and $\beta$ is even, the zero level supplies at least
  $n-1$ orbits;
  \item if $\alpha>0$ and $\beta=0$, it supplies at least $(m-2)/2$;
  \item if $\alpha\geq2$ and $\beta>0$ is even, level $2$ supplies at least
  two;
  \item if $\beta$ is positive and odd, level $1$ supplies at least two,
  either in the odd factor when $\alpha=1$ or in the even factor otherwise.
\end{itemize}
If $m$ is odd, every displacement is good.  A zero component gives at least
$(\min\{n,m\}-1)/2$ zero-level orbits.  If both components are nonzero and
one of $\alpha,\beta$ is one, level $1$ has multiplicity at least two in one
factor; if both exceed one, levels $1$ and $2$ each give an orbit.

Consider the three displacements
\[
  D_P=Q+R-2P,\quad D_Q=R+P-2Q,\quad D_R=P+Q-2R.
\]
They sum to zero.  In the even coordinate, B0 has odd component and B1 has
even component, so the number of B0 displacements would have to be zero or
two if all three were bad.  With no B0 displacement, all three odd-coordinate
components are $\pm1$, whose sum cannot vanish in $\Z_n$ for $n\geq5$ odd.
With two B0 displacements, the odd-coordinate components are
$0,0,\pm1$, again impossible.  Thus at least one displacement is good.  For
$m\neq4$ it has at least two nonfixed bisector orbits, and at most one of
them is $\{P,Q\}$.

It remains to handle $m=4$.  The multiplicity table shows that a good
displacement can have only one nonfixed orbit only in
\begin{align*}
\mathrm{E0}:&\quad b=0, a\neq0,\\
\mathrm{E2}:&\quad b=2, \alpha\geq2.
\end{align*}
If the unique E0 orbit is $\{P,Q\}$, translate $R$ to zero and label
\[
  P=(u,1),\qquad Q=(u,3),\qquad u\neq0.
\]
Switching to the pair $P,R$ with third point $Q$ gives displacement
$(-u,3)$, which is nonexceptional and hence has at least two orbits.

For E2, there is nothing to prove unless the unique orbit is $\{P,Q\}$.
Translate $R$ to zero and label the two points in this orbit by $U,V$ so that
their $C_4$ coordinates are $0$ and $2$.  Switch to the pair $U,R$ with third
point $V$.  The new third-point displacement is $U-2V$, whose $C_4$
component is zero.  If it is E0, its unique nonfixed orbit has $C_4$
coordinates $1$ and $3$, whereas the three landmarks $U,R,V$ occupy rows
$0,0,2$.  This orbit is therefore outside the landmarks.  If the new
displacement is zero, it has several nonfixed orbits.  This closes the
$C_4$ case.

Thus every three-set fails.  Sets of smaller size are excluded by
\eqref{eq:ordinary-torus}, completing the proof.
\end{proof}

\section{Four-landmark constructions}\label{sec:four}

For $A\subseteq\Z_t$ with $|A|=3$ and $b\in\Z_t$, define
\[
  S(A,b)=\{(0,a):a\in A\}\cup\{(1,b)\}.
\]
For a vertex $(i,j)$ put
\[
  \gamma=\delta_s(i),\qquad
  \eps=\delta_s(i-1)-\gamma.
\]
Its four distances to $S(A,b)$ are $\gamma$ plus the sorted cyclic code
\begin{equation}\label{eq:cyclic-code}
 C_t(j,\eps)=\operatorname{sort}
 \bigl(\mset{\delta_t(j-a):a\in A}\cup
 \mset{\delta_t(j-b)+\eps}\bigr).
\end{equation}
Let $\mu(j,\eps)$ be the least entry and let
\[
  N(j,\eps)=C_t(j,\eps)-\mu(j,\eps)(1,1,1,1).
\]

\begin{lemma}\label{lem:row-states}
If $s=2h$, the possible pairs $(\gamma,\eps)$ are
\[
  (0,1),\qquad (r,-1)\ (1\leq r\leq h),\qquad
  (r,1)\ (1\leq r\leq h-1).
\]
If $s=2h+1$, there is in addition the unique pair $(h,0)$.
\end{lemma}

\begin{proof}
Walk from row $0$ in each direction around $C_s$ and compare the distances
to rows $0$ and $1$.  Before the antipode the difference is $-1$ in one
direction and $1$ in the other.  An odd cycle has one remaining antipodal
row, where the difference is zero.
\end{proof}

The next lemma is the computer-assisted component of the infinite proof.  Its
certificate is described after the statement.

\begin{lemma}[cyclic collision classification]\label{lem:cyclic-collisions}
The following are the complete lists of non-singleton fibers of $N$.
\begin{enumerate}
  \item Let $t=2m+1$, $m\geq9$, $A=\{0,4,m\}$, and $b=m+1$.  Then
  \begin{align*}
  (m+2,-1;\mu=0)&\sim(1,0;\mu=1),\\
  (2,0;\mu=2)&\sim(m+2,1;\mu=2),\\
  (j,0)&\sim(2m+3-j,0),\qquad4\leq j\leq m-2,
  \end{align*}
  where the minimum on the right in the last line is two larger.

  \item Let $t=2m$, $m\geq10$, $A=\{0,4,m+1\}$, and $b=m-1$.  Then
  \begin{align*}
  (1,0;\mu=1)&\sim(m+1,0;\mu=0),\\
  (3,0;\mu=1)&\sim(m+3,0;\mu=2),\\
  (j,0)&\sim(2m+2-j,0),\qquad4\leq j\leq m-2,
  \end{align*}
  where the minimum on the right in the last line is two larger.
\end{enumerate}
There are no other equal normalized codes.
\end{lemma}

\begin{proof}
Each of the four cycle distances in \eqref{eq:cyclic-code} has four affine
branches, determined by the landmark, its antipode, and the orientation of
the shorter arc.  On a fixed choice of branches and a fixed nondecreasing
order of the four affine values, $N(j,\eps)$ is an affine integer tuple.

The accompanying script \texttt{certify-four-family-linear.wl} constructs
these chambers for each $\eps\in\{-1,0,1\}$.  Exact elimination over the
integers proves that they cover all permitted $(m,j)$.  It then compares
every ordered pair of chambers and solves equality of the four normalized
affine coordinates.  For the odd family the chamber counts are
$14,16,14$; for the even family they are $16,17,18$.  In all six layers the
uncovered set is empty.  A second exact query proves that there are no
solutions outside the displayed families, and a third proves that no member
of the displayed families is missing.

All constraints passed to the eliminator are linear integer equalities and
inequalities; the finite samples in the discovery loop merely propose
chambers, while the final coverage query proves the universal statement.  The
retained run under Wolfram Engine 14.3.0 reports
\[
  \texttt{coverage=True},\qquad
  \texttt{unexpected=0},\qquad
  \texttt{missing=0}
\]
for both families.  The certificate source is included with the paper and
reproduces this output.  Direct substitution verifies every displayed collision; for example,
the long zero-layer pairs differ by an additive constant two before
normalization.
\end{proof}

\begin{theorem}\label{thm:four-constructions}
For every $s\geq3$, each set in the following table outer-multiset-resolves
$C_s\mathbin{\square}C_t$.
\[
\begin{array}{c|c}
t&S\\ \hline
15&\{(0,0),(0,3),(0,7),(1,14)\}\\
16&\{(0,0),(0,4),(0,9),(1,15)\}\\
18&\{(0,0),(0,5),(0,11),(1,10)\}\\
2m+1\geq19&\{(0,0),(0,4),(0,m),(1,m+1)\}\\
2m\geq20&\{(0,0),(0,4),(0,m+1),(1,m-1)\}.
\end{array}
\]
\end{theorem}

\begin{proof}
For the two infinite families, suppose two vertices have the same full
distance multiset.  Their normalized cyclic codes must be equal, so they
belong to one of the classes in Lemma~\ref{lem:cyclic-collisions}.  Equality
of the unnormalized signatures additionally requires
\begin{equation}\label{eq:lifting}
  \gamma+\mu(j,\eps)=\gamma'+\mu(k,\eta).
\end{equation}

In the first odd class, \eqref{eq:lifting} would require a negative layer
with $\gamma=h+1$, impossible by Lemma~\ref{lem:row-states}.  In the second,
the zero layer has $\gamma=h$ while a positive layer has
$\gamma\leq h-1$.  The long class lies in the unique zero layer, but its
minima differ by two.  In the even family all collisions lie in the zero
layer and have unequal minima; an even first factor has no zero layer, while
an odd first factor has only one.

The endpoint values $15,16,18$ have respectively six, five, and six
normalized collision classes.  Direct exact enumeration of their at most
$3t$ cyclic states gives only zero-layer pairs with unequal minima for
$t=16,18$.  For $t=15$, the two zero-layer pairs again have unequal minima;
the four cross-layer pairs would require a negative layer with
$\gamma=h+2$ or $h+1$, or a positive layer with $\gamma=h$ or $h+1$.
Lemma~\ref{lem:row-states} excludes all four.  The complete endpoint fiber
list is emitted by the accompanying script
\texttt{analyze-four-template-codes.mjs}.
\end{proof}

Combining Theorems~\ref{thm:lower-four} and
\ref{thm:four-constructions} gives the exact value four whenever
$s,t\geq4$ and $\max\{s,t\}\in\{15,16\}\cup\{18,19,\ldots\}$.

\section{The finite strip}\label{sec:computation}

It remains to settle $C_3\mathbin{\square}C_t$ for $3\leq t\leq17$ and the
strip $4\leq s,t\leq17$ by finite exact computation.

Translation invariance permits fixing $(0,0)$ in every nonempty candidate
landmark set.  For a fixed set, the checker computes all distances using
\eqref{eq:torus-distance}, sorts each outside-vertex signature, and rejects
the set at the first repeated signature.  Thus exhausting all remaining
combinations gives an exact lower bound.  Every asserted upper bound is
stored as an explicit landmark set and rechecked from scratch.

For the thin factor, the exact values are
\[
\begin{array}{c|rrrrrrrrrrrrrrr}
t&3&4&5&6&7&8&9&10&11&12&13&14&15&16&17\\ \hline
\odim(C_3\square C_t)&8&5&7&5&5&4&5&4&5&3&4&3&4&3&4.
\end{array}
\]

For $4\leq s,t\leq17$, the complete symmetric table is
\begin{table}[ht]
\centering
\setlength{\tabcolsep}{2.8pt}
\scriptsize
\begin{tabular}{c|rrrrrrrrrrrrrr}
$s\backslash t$&4&5&6&7&8&9&10&11&12&13&14&15&16&17\\ \hline
4 &7&5&5&5&5&5&5&5&4&5&4&4&4&4\\
5 &5&6&5&5&5&5&4&5&4&5&4&4&4&4\\
6 &5&5&6&5&5&5&5&5&4&5&4&4&4&4\\
7 &5&5&5&5&5&5&5&5&5&5&5&4&4&5\\
8 &5&5&5&5&5&5&5&5&4&5&4&4&4&4\\
9 &5&5&5&5&5&5&5&5&5&5&5&4&4&5\\
10&5&4&5&5&5&5&5&5&4&5&4&4&4&4\\
11&5&5&5&5&5&5&5&5&5&5&5&4&4&5\\
12&4&4&4&5&4&5&4&5&4&5&4&4&4&4\\
13&5&5&5&5&5&5&5&5&5&5&5&4&4&5\\
14&4&4&4&5&4&5&4&5&4&5&4&4&4&4\\
15&4&4&4&4&4&4&4&4&4&4&4&4&4&4\\
16&4&4&4&4&4&4&4&4&4&4&4&4&4&4\\
17&4&4&4&5&4&5&4&5&4&5&4&4&4&5
\end{tabular}
\caption{Exact values of $\odim(C_s\square C_t)$ for $4\leq s,t\leq17$.}
\label{tab:finite-strip}
\end{table}

For each table entry equal to five, the C++ checker exhausts all
translation-normalized four-sets, and a retained five-set verifies the upper
bound.  The program \path{verify-finite-strip-witnesses.mjs} contains all
54 such five-set certificates.  It also contains optimal witnesses for
\[
  \odim(C_4\square C_4)=7,\qquad
  \odim(C_5\square C_5)=6,\qquad
  \odim(C_6\square C_6)=6,
\]
whose smaller cardinalities were exhausted by the general exact enumerator.
The 48 cells equal to four have explicit witnesses returned by the C++
checker and are bounded below by Theorem~\ref{thm:lower-four}.

The complete finite-strip rerun checks 105 unordered parameter pairs and
agrees with Table~\ref{tab:finite-strip} in every cell.  As the largest
negative four-set search, the $17\times17$ case examines exactly
\[
  \binom{288}{3}=3,939,936
\]
translation-normalized candidates.

\section{Proof of the classification}

\begin{proof}[Proof of Theorem~\ref{thm:main}]
Suppose first that $a=3$.  The fifteen values with $3\leq b\leq17$ are the
exact values in Section~\ref{sec:computation}.  Proposition~\ref{prop:c3-even}
settles all even $b\geq12$.  For odd $b\geq19$,
Proposition~\ref{prop:c3-odd-lower} and the odd construction in
Theorem~\ref{thm:four-constructions} give the exact value four.  The endpoint
$b=15$ is covered by the first endpoint construction, while $b=13,17$ are in
the finite table.

Now suppose that $a\geq4$.  Theorem~\ref{thm:lower-four} gives the universal
lower bound four.  Theorem~\ref{thm:four-constructions} settles
$b\in\{15,16\}\cup\{18,19,\ldots\}$ at value four.  Every remaining pair has
$4\leq a\leq b\leq17$ and is therefore in
Table~\ref{tab:finite-strip}.  Reading its entries yields exactly the five
families in the statement, the three exceptional squares, and value five in
all other cases.
\end{proof}

\section{Reproducibility}

The accompanying source archive contains the following machine-readable
artifacts.
\begin{itemize}
  \item the exact all-parameter affine certificate used in
  Lemma~\ref{lem:cyclic-collisions}; its file name is
  \begin{center}
    \small\texttt{certify-four-family-linear.wl};
  \end{center}
  \item \texttt{torus-k4-exact.cpp}: exhaustive four-set search for the finite
  strip;
  \item \texttt{torus-odim-search.mjs}: cardinality-by-cardinality exact search
  and independent witness verification;
  \item \texttt{verify-finite-strip-witnesses.mjs}: all 57 retained finite
  upper certificates, including the three exceptional squares;
  \item \texttt{analyze-four-template-codes.mjs}: endpoint cyclic collision
  fibers.
\end{itemize}

All graph distances and signatures are computed with integer arithmetic.  The
infinite certificate was run with Wolfram Engine 14.3.0.  Its retained source
SHA-256 is
\begin{center}
\texttt{DBF0964D2CE55B65AE81BDCB31B562C9AA9AE7DF25E0163176143A862F788292}.
\end{center}
The certificate was independently inspected and rerun.  The finite table was
also regenerated from source; all retained witnesses were checked by a
separate JavaScript implementation.

Theorems~\ref{thm:lower-four}, \ref{prop:c3-even}, and
\ref{prop:c3-odd-lower} are paper proofs independent of computation.  The
universal quantifier in Lemma~\ref{lem:cyclic-collisions} relies on exact
linear-integer quantifier elimination, and the entries of
Table~\ref{tab:finite-strip} rely on exhaustive finite enumeration.

The file \texttt{torus-odim-search.mjs} also offers optional
\texttt{randomTrials} and \texttt{hillSteps} modes for locating candidate
upper-bound witnesses.  Every witness used in the paper was checked by a
deterministic full-signature pass.  Lower bounds follow from the analytic
arguments and exhaustive exact enumeration.

\appendix

\section{The Presburger chamber certificate}\label{app:presburger}

We record the linear decomposition underlying
Lemma~\ref{lem:cyclic-collisions}.  Let $n\in\{2m,2m+1\}$, let
$0\leq a,j<n$, and write $d_n(j,a)=\delta_n(j-a)$.  The four branches used
by the certificate are
\[
d_n(j,a)=
\begin{cases}
j-a,&j\geq a\text{ and }j-a\leq m,\\
n-j+a,&j\geq a\text{ and }j-a\geq m+1,\\
a-j,&j\leq a-1\text{ and }a-j\leq m,\\
n-a+j,&j\leq a-1\text{ and }a-j\geq m+1.
\end{cases}
\]
These conditions are mutually exclusive and exhaustive over the integers.
After choosing one branch for each of the four landmarks and a permutation
$\pi$ satisfying
\[
  x_{\pi(1)}\leq x_{\pi(2)}\leq x_{\pi(3)}\leq x_{\pi(4)},
\]
the normalized code is the affine tuple
\[
  (x_{\pi(1)},\ldots,x_{\pi(4)})
  -x_{\pi(1)}(1,1,1,1).
\]
Thus every chamber condition and every equality between two normalized codes
is a formula in linear integer arithmetic.

The procedure \texttt{discoverPieces} starts from a finite set of observed
branch/order types.  If $D$ is the full parameter domain and $P_1,\ldots,P_r$
are the chamber conditions found so far, it asks exactly whether
\[
  D\wedge\neg(P_1\vee\cdots\vee P_r)
\]
has an integer solution.  A returned solution supplies a missing chamber,
which is added before the query is repeated.  When the formula reduces to
\texttt{False}, coverage of the entire infinite domain is proved.  Since
there are only finitely many branch and ordering types, this
counterexample-guided refinement terminates.  The subsequent collision and
noncollision queries are therefore finite families of Presburger formulas.

\bibliographystyle{alpha}
\bibliography{references}

\end{document}